\documentclass[a4,psfig,english,12pt,epsf,portrait]{article}
\usepackage{amsmath,amsfonts,latexsym,amscd,amssymb}
\usepackage{epsfig}
\usepackage{amsmath}
\usepackage{amsthm}
\usepackage{amssymb}
\usepackage{enumerate}
\usepackage{tikz}
\usepackage{color}
\usepackage{hyperref}
\hypersetup{hidelinks}
\definecolor{purple}{rgb}{0.65, 0, 1}
\definecolor{orange}{rgb}{1,.5,0}
\definecolor{brown}{rgb}{.9,.73,.26}

\usepackage{natbib}

\newtheorem{theorem}{Theorem}[section]
\newtheorem{remark}{Remark}[section]
\newtheorem{proposition}[theorem]{Proposition}
\newtheorem{definition}{Definition}[section]
\newtheorem{lemma}[theorem]{Lemma}

\numberwithin{equation}{section}

\def\R{{\mathbb {R}}}

\newcommand{\abs}[1]{\lvert#1\rvert}
\usepackage{color}

\def\R{\mathbb{R}}
\def\N{\mathbb{N}}

\title{\bf{Local and global solvability of the Grushin heat equation with memory}}

\author{Ahmad Z. Fino \& Arlucio Viana}

\date{}
	
\begin{document}
\maketitle



\begin{abstract} 
	In this work, we investigate the solvability of a heat equation involving the Grushin operator. The equation is perturbed by two nonlinear reaction terms, one of which includes a memory component, introducing nonlocal effects in time. We first establish local well-posedness results in both Lebesgue spaces and in the space of continuous functions vanishing at infinity. Furthermore, we develop and apply comparison principles that allow us to derive sufficient conditions for the global existence of solutions, depending on the relative strength and nature of the nonlinearities involved. In particular, we highlight how the presence of the memory term influences global solvability. Our results contribute to the understanding of parabolic equations with degenerate operators and complex nonlinear interactions.

\end{abstract}

\section{Introduction}

Global existence of solutions for the heat equation with power-type nonlinearities goes back to Fujita's paper \cite{Fujita}, where he proves existence and nonexistence results for the semilinear heat equation
\begin{equation}\label{sheat}
	\begin{array}{ll}
		u_t(x,t) = \Delta u(x,t) + |u(x,t)|^{\rho-1}u(x,t),&\quad  \mbox{in}\ \R^N \times (0,\infty),\\
		{}\\
		u(x,0)=u_0(x),&\quad  \mbox{in}\ \R^N ,
	\end{array}
\end{equation} 
The critical case $\rho = 1+\frac{2}{N}$ was resolved by Weissler \cite{Weissler}. Since then many works have been devoted to study results of this type, and many of them predicted the critical Fujita exponent from the critical space, which is given by a scaling computation. Later on, Brézis and Cazenave treated the local well-posedness for \eqref{sheat} in Lebesgue spaces. Nonetheless, nonlocal nonlinearities play an important role in the analysis of the evolution equations, for example, dissociating the critical Fujita exponent to the one given by scaling procedures. In fact, Cazenave, Dickstein and Weissler \cite{CDW} proved that the critical exponent for 
$$ u_t-\Delta u=\int_0^t(t-s)^{-\gamma}|u|^{p-1} u(s) d s \quad x \in \mathbb{R}^N, t>0 ,$$
where $0 \leq \gamma<1$, $p>1$, is $p_*=\max \left\{\frac{1}{\gamma}, p_\gamma\right\}$, where
$$
p_\gamma=1+\frac{2(2-\gamma)}{(N-2+2 \gamma)_{+}} ,
$$
while the exponent predicted by the scaling is $p_{\mathrm{sc}}=1+\frac{4-2 \gamma}{N}$. Fino and Kirane \cite{Fino-Ki-12} generalized that result by replacing the Laplacian with the fractional Laplacian. This type of investigation is became noticeable, because equations with distinct operators often have different scaling exponents, which bring new difficulties inherent to the equations, see e.g. \cite{Cas-Loy-19,Dao-Fino-22,Ki-La-T-05,Loy-P-14}. 

At the same time, sub-elliptic operators are of great scientific interest partly for they are naturally connected with anisotropic diffusion phenomena. Examples of such operators are the Heisenberg sub-Laplacian, Kohn Laplacians, Grushin operator. Each one of them carries its own geometric properties that highly influence its analysis, see e.g. \cite{Alves-G-Lo-24,Calin-11, Chang-Li-12,Gaveau-77,Goldstein-Kogoj, Kogoj-Lan-12,Wu-15}. Heat kernels of those type of operators have been studied, for instance, in \cite{ Ba-Furutani-15, Arous-89, Chang-Li-15, Gaveau-77}. We can find a good survey of techniques to find the heat kernel of sub-elliptic operators in Callin's \textit{et al} monograph \cite{Calin-11}. 

Recently, there is an increasing interest for nonlinear parabolic equations with sub-elliptic operators, see e.g. the recent papers \cite{Fino-Ruz-Tor-24, Hi-Oka-24,J-Ki-S-16,RuzY-22}. Those papers lead mostly with sub-Laplacian operators and do not include the Grushin
\begin{equation*}\label{Grushiop}
	\Delta_{\mathcal{G}}=\dfrac{1}{2}\left(\Delta_{x}+|x|^2\Delta_{y}\right), 
\end{equation*}
where $\Delta_x, \Delta_y$ denote the classical Laplacian in the variables $x\in\R^N$ and $y\in\R^k$, respectively. In this direction, Oliveira and Viana \cite{Oli-Vi-23} proved the existence, uniqueness, continuous dependence and blowup alternative of local mild solutions for 
\begin{equation}\label{sheatG}
	\begin{array}{ll}
		u_t = \Delta_\mathcal{G} u + |u|^{\rho-1}u,&\quad \mbox{in}\  \R^{N+k}\times (0,\infty) ,\\ 
		u(0)=u_0,&\quad \mbox{in}\ \R^{N+k},
	\end{array}
\end{equation} 
with initial conditions in Lebesgue spaces. Through the analysis of the heat kernel associated to the Grushin operator, they obtain the existence of global solutions in the special case of $u_0\in L^{\frac{ N}{2}(\rho-1)}(\R^{N+k})$ with the sufficiently small norm. Later on, Kogoj, Lima and Viana \cite{Kogoj-Lima-Viana-24} by working in Marcinkiewcz spaces, they allowed larger (in Lebesgue norm) initial data to be taken into account in order to obtain global solutions. More precisely, they gave sufficient conditions to the existence and uniqueness of mild solutions for \eqref{sheatG}, with initial conditions in the critical Marcinkiewcz space $L^{p,\infty}(\R^{N+k})$, with 
\begin{equation}\label{p}
	p = \frac{N+2k}{2}(\rho-1) .
\end{equation}
Moreover, they obtained the existence of positive, symmetric and self-similar solutions. The number $N+2k$ is the so-called homogeneous dimension attached to Grushin's operator (see \cite{Kogoj-Lan-18}).

Motivated by those facts, we recognize the relevancy of Fujita exponents for  heat equations with the Grushin operator and memory: 
\begin{equation}\label{*}
	\left\{\begin{array}{ll}
		\displaystyle {u_{t}-\Delta_{\mathcal{G}}
			u = k_1 \int_0^t(t-s)^{-\gamma}\abs{u}^{p_1-1}u(s)\,\mathrm{d}s} + k_2|u|^{p_2-1}u, &\displaystyle {t>0,z\in {\mathbb{R}}^{N+k},}\\
		{}\\
		\displaystyle{u(x,0)=  u_0(z),\qquad\qquad}&\displaystyle{z=(x,y)\in {\mathbb{R}}^{N+k},}
	\end{array}
	\right.
\end{equation}
where $u_0 \in X$, $k_1,k_2\in \R$, and $p_1,p_2>1$.  Here, $X$ denotes either $L^q(\mathbb{R}^{N+k})$, with $1\leq q< \infty$, or $C_0(\mathbb{R}^{N+k})$. In this manuscript we address to the problem of existence, uniqueness, continuous dependence  and a blow-up alternative for local solutions of \eqref{*} as well as the existence of positive global solution for \eqref{*}.

The local existence results are similar to those given in \cite{BrezisCaz} and \cite{CDW}, in the case of the classical diffusion equation and the heat equation with memory in the source, respectively. The equation \eqref{*} without memory, that is, $k_1=0$, was studied in Oliveira and Viana \cite{Oli-Vi-23}.

We prove results on the existence of local solutions for \eqref{*}, when initial conditions are either in Lebesgue spaces or are continuous functions vanishing at infinity. The results on global existence of positive solutions always assume that the initial condition is in $C_0(\mathbb{R}^{N+2k})$, but it is used mostly to get prove regularity of the mild solutions. Weaker formulations for them would require only the initial conditions to be in Lebesgue spaces the most of our proofs can be used.

The manuscript is organized in the following way. We bring the main already known tools in Section \ref{pre}. In Section \ref{local}, we state and prove the results local well-posedenss result and the comparison principle. Finally, we give sufficient conditions for the global existence of positive solutions in Section \ref{global}.

\section{Preliminaries}\label{pre}

In this section we gather definitions and results used in our proofs. They are related to the heat kernel associated to the Grushin operator, the notion of solutions, tools from the Fractional Calculus and integrodifferential inequalities.

\subsection{Heat kernel}

We begin by recalling the \textit{explicit} form of the heat kernel associated with the Grushin operator (see, e.g., \cite[Proposition~2.5]{Oli-Vi-23} or \cite{Garofalo-Trallli-22}):
\begin{equation}\label{HK}
	K(x, x_0,y;t) = \frac{1}{(2\pi)^{\frac{N+2k}{2}}} \int_{\R^k} \left(\frac{|\xi|}{\sinh(|\xi| t)}\right)^{\frac{N}{2}} e^{i\xi  \cdot y- \frac{|\xi|}{2}\left((|x|^2 + |x_0|^2)  \coth(|\xi| t) -2x\cdot x_0 \mathrm{csch}(|\xi|t) \right)} \mathrm{d}\xi, 
\end{equation}
for $(x, x_0,y) \in \R^{2N+k}$, and  $t>0$. The kernel given above is smooth, i.e., $C^\infty(\R^{2N+k}\times(0,\infty))$.\\
Furthermore, the following results either are collected from \cite{Bi-Bra-23,Kogoj-Lima-Viana-24, Oli-Vi-23} or follow from them. We keep in mind that the Schwartz space $\mathcal{S}(\R^{N+k})$ is dense in $L^p(\mathbb{R}^{N+k})$ for $1\leq p< \infty$, and in $C_0(\mathbb{R}^{N+k})$. Here, $C_0(\R^{N+k})$ denotes the space of continuous functions vanishing at infinity, endowed with the supremum norm $\|\cdotp\|_\infty$.\\

\begin{proposition}\label{Heat}${}$
	If u denotes the solution of the heat equation with the Grushin operator,
	\begin{equation}\label{cauchy}
		\begin{array}{ll}
			\partial_{t}u-\Delta_{\mathcal{G}}u=0,&\qquad \mathrm{in}\  \R^{N+k}\times (0,\infty),
			\\{}\\ u(0)=u_0\in \mathcal{S}(\R^{N+k}),&\qquad \mathrm{in}\  \R^{N+k},
		\end{array}
	\end{equation}
	then we have
	\begin{equation}\label{solinear}
		u(x,y,t)=	\int_{\R^{N+k}} K(x,w,y-z;t) u_0(w,z)\,dw\,dz,\qquad (x,y)\in \R^{N+k},\,\,t>0,
	\end{equation} 
	where $(x, x_0,y;t)\in\R^{2N+k}\times(0,\infty)\mapsto K(x, x_0,y;t) $ denotes the heat kernel associated to $-\Delta_{\mathcal{G}}$. Moreover, $u\in C([0,\infty)\times\R^{N+k})\cap C^1((0,\infty),C^2(\R^{N+k}))$.
\end{proposition}

\begin{proposition}\label{properties}
	There is a unique semigroup $(S_{\mathcal{G}}(t))_{t\geq0}$ associated to $\Delta_{\mathcal{G}}$ and satisfying the following properties:
	\begin{itemize}
		\item[$(i)$]  For all $\varphi\in L^p(\R^{N+k})$, $1\leq p\leq \infty$, we have
		$$S_{\mathcal{G}}(t)\varphi(x,y)=\int_{\R^{N+k}} K(x,w,y-z;t) \varphi(w,z)\,dw\,dz,\qquad (x,y)\in \R^{N+k},\,\,t>0.$$
		\item[$(ii)$] $(S_\mathcal{G}(t))_{t\geq0}$ is a contraction semigroup on $L^p(\mathbb{R}^{N+k})$, $1\leq p\leq \infty$, which is strongly continuous for $1\leq p<\infty$. 
		\item[$(iii)$] $(S_\mathcal{G}(t))_{t\geq0}$  is a strongly continuous semigroup on $C_0(\mathbb{R}^{N+k})$.
		\item[$(iv)$] For every $v\in X$, where $X$ is either $L^p(\mathbb{R}^{N+k})$ for $1\leq p< \infty$ or $C_0(\mathbb{R}^{N+k})$, the map $t\longmapsto S_\mathcal{G}(t)v$ is continuous from $[0,\infty)$ into $X$.
		\item[$(v)$] $\displaystyle\int_{\mathbb{R}^{N+k}}K(x,0,y;t)  \,dx\,dy=1$, for all $t>0$.
		\item[$(vi)$] The heat kernel satisfies
		$$K(x,x_0,y;t)>0,\qquad \textrm{for all}\,\, (x, x_0,y;t)\in\R^{2N+k}\times(0,\infty).$$
		Hence, $S_\mathcal{G}(t)\varphi\geq 0$ whenever $\varphi\geq0$, for all $t\geq0$.
		\item[$(vii)$] $K(rx,rx_0,r^2y;r^2t)=r^{-(N+2k)}K(x,x_0,y;t)$, for all $(x, x_0,y;t)\in\R^{2N+k}\times(0,\infty)$, $r>0$.
		\item[$(viii)$]  For all $(x, x_0,y;t)\in\R^{2N+k}\times(0,\infty)$, we have
		$$K(x,x_0,-y;t)=K(x,x_0,y;t)\quad\mathrm{and}\quad K(x,x_0,y;t)=K(x_0,x,y;t).$$
	\end{itemize}
\end{proposition}

\begin{lemma}[$L^p-L^q$ estimate]\label{Lp-Lqestimate} 
	Suppose $1\leq p\leq q\leq \infty$. Then there exists a positive constant $C$ such that for every $f\in L^p(\mathbb{R}^{N+k})$, the following inequality holds:
	\begin{equation}\label{semigroup}
		\|S_{\mathcal{G}}(t)f\|_q\leq C\,t^{-\frac{N+2k}{2}(\frac{1}{p}-\frac{1}{q})}\|f\|_p,\qquad t>0.
	\end{equation}
\end{lemma}

We need the following lemma.
\begin{lemma}\label{lemmaSg}
	Let $1\leq q<r<\infty$, and $\alpha = \frac{N+2k}{2}\left(\frac{1}{q} - \frac{1}{r}\right)$. Then
	\begin{equation}
		\lim\limits_{t\rightarrow 0^+}t^{\alpha}\|S_\mathcal{G}(t)\varphi\|_{L^r}=0,
	\end{equation}
	for any $\varphi\in L^q(\R^{N+k})$.
\end{lemma}

\begin{proof} Given $\varphi \in L^q(\R^{N+k})$. Then, by density argument, there exists a sequence $(\varphi_n)_{n \in \N}\subset C_c^\infty(\R^{N+k})$ converging to $\varphi$ in $L^q(\R^{N+k})$. Using the semigroup decay \eqref{semigroup}, we have
	\begin{eqnarray*}
		t^\alpha\|S_\mathcal{G}(t)\varphi\|_{L^r}&\leq& t^\alpha\|S_\mathcal{G}(t)(\varphi-\varphi_n)\|_{L^r}+t^\alpha\|S_\mathcal{G}(t)\varphi_n\|_{L^r}\\
		&\leq&C\|\varphi-\varphi_n\|_{L^q}+t^\alpha\|\varphi_n\|_{L^r}\\
	\end{eqnarray*}
	for all $n\in \N, t>0$.
	Taking the limit as $t \rightarrow 0$ , we get
	\begin{equation*}
		\lim\limits_{t\rightarrow 0^+}t^{\alpha}\|S_\mathcal{G}(t)\varphi\|_{L^r}\leq C\|\varphi-\varphi_n\|_{L^q},
	\end{equation*}
	and the result follows by letting $n \rightarrow \infty$.
\end{proof}

\section{Local Existence and the Comparison Principle}\label{local}

In this section, we establish the local well-posedness of mild solutions to problem \eqref{*}, addressing both singular and regular initial data. Specifically, we prove the local existence and provide a blow-up alternative for mild solutions corresponding to initial data either $u_0 \in L^q(\R^{N+k})$ or $u_0\in C_0(\mathbb{R}^{N+k})$.

We begin by introducing the notion of a mild solution to problem \eqref{*} .
\begin{definition} Let $X$ denote either $L^q(\mathbb{R}^{N+k})$, with $1\leq q< \infty$, or $C_0(\mathbb{R}^{N+k})$. A function
	$$u\in L^\infty([0,T],X)$$ 
	is called an {\bf $X$-mild solution} of $\eqref{*}$ if it satisfies the integral equation
	\begin{equation}\label{branda}
		u(t)=S_{\mathcal{G}}(t)u_0 + k_1 \int_{0}^{t}S_{\mathcal{G}}(t-s) \int_0^s(s-\tau)^{-\gamma}\abs{u(\tau)}^{p_1-1}u(\tau) \, \mathrm{d}\tau \,\mathrm{d}s + k_2\int_{0}^{t}S_{\mathcal{G}}(t-s) |u(s)|^{p_2-1}u(s)\mathrm{d}s.
	\end{equation}
\end{definition}


\subsection{Local existence with singular initial condition}
Let $T>0$ be fixed. We define the Banach space
$$E=L^\infty((0,T);L^q(\R^{N+k}))\cap L^\infty_{loc}((0,T), L^{r_1}(\R^{N+k}))\cap L^\infty_{loc}((0,T), L^{r_2}(\R^{N+k})),$$
where $r_i=p_i q$ for $i=1,2$.  Based on this, we introduce the Banach space 
\begin{equation*}
	X_T=\{u \in E:\,\|u\|_{X_T}<+\infty\},
\end{equation*}
which will serve as the framework for our fixed point argument. The norm on $X_T$ is defined by
$$\|u\|_{X_T}= \sup_{t \in (0,T)}\|u(t)\|_{L^{q}}+\max_{i=1,2} \sup_{t \in (0,T)}t^{\alpha_i}\|u(t)\|_{L^{r_i}},$$
where the exponents $\alpha_i$ are given by 
$$\alpha_i=\frac{N+2k}{2p_i q}\left(p_i-1\right),\qquad i=1,2.$$
\begin{theorem}[Local existence with singular initial condition]\label{teoexun}
	Let $p_i>1$ for $i=1,2$, and let $u_0 \in L^q(\R^{N+k})$ with $1\leq q <\infty$, satisfying
	$$q>\max\limits_{i=1,2}\frac{N+2k}{2}(p_i-1).$$
	Then there exists a time $T=T(u_0)>0$ such that problem \eqref{*} possesses an $L^q$-mild solution $u$, which is unique in $X_T$. Moreover, the following properties hold:
	\begin{itemize}
		\item[$\mathrm{(i)}$] Smoothing effect and continuous dependence,
		\begin{equation}\label{contdep}
			\sup_{t \in [0,T]} \|u(t)-v(t)\|_{L^q}+\max_{i=1,2} \sup_{t \in (0,T)}t^{\alpha_i}\|u(t)-v(t)\|_{L^{r_i}} \leq C\|u_0-v_0\|_{L^q} ,
		\end{equation}
		whenever $u$ and $v$ are $L^q$-mild solutions of $\eqref{*}$ with time of existence $T=\min\{T(u_0),T(v_0)\}$ and initial data $u_0$ and $v_0$, respectively.
		\item[$\mathrm{(ii)}$]  Decay of solution,
		\begin{equation}\label{conv}
			\lim_{t\rightarrow0^+} t^{\alpha_i}\|u(t)\|_{L^{r_i}} =0,\quad\hbox{for all}\,\,i=1,2.
		\end{equation}
		\item[$\mathrm{(iii)}$] If in addition $q\geq \max\limits_{i=1,2}  p_i$, then the mild solution  $u$ is also unique in $C([0,T];L^q(\R^{N+k}))$. In this case, problem \eqref{*} possesses a maximal mild solution $u\in C([0,T_{\max});L^q(\R^{N+k}))$ with $u(0)=u_0$, where $T_{\max}=T_{\max}(u_0)\leq\infty$. Furthermore, either $T_{\max}=\infty$, or else $T_{\max}<\infty$ and
		\begin{equation}\label{limblowup}
			\lim_{t\rightarrow T_{\max}^-}\|u\|_{L^\infty((0,t),L^q(\mathbb{R}^{N+k}))}=+\infty.
		\end{equation}
		\item[$\mathrm{(iv)}$] If $u_0\geq0$ and $k_1,k_2\geq0$, then then the mild solution $u\geq0$ remains nonnegative for all $t\in[0,T]$.
	\end{itemize}
\end{theorem}

\begin{proof} The proof is divided into several steps.\\
	{\it Step 1. Fixed-point argument.} Let $T>0$ and choose $R\geq C \|u_0\|_{L^q}$ with $C>0$ is given in \eqref{semigroup}. In order to use the Banach fixed-point theorem, we introduce the following nonempty complete metric space
	\begin{equation}\label{norm}
		B_T(R)=\{u\in X_T:\,\,\|u\|_{X_T}\leq 4 R\},
	\end{equation}
	equipped with the distance
	$$d(u,v)=\max_{i=1,2} \sup_{t \in (0,T)}t^{\alpha_i}\|u(t)-v(t)\|_{L^{r_i}} .$$
	For $u\in B_{T}(R)$, we define $\Lambda(u)$ by
	\begin{equation}\label{solop}
		\Lambda(u)(t):=S_{\mathcal{G}}(t)u_0 +k_1 \int_{0}^{t}S_{\mathcal{G}}(t-s) \int_0^s(s-\tau)^{-\gamma}\abs{u(\tau)}^{p_1-1}u(\tau) \, \mathrm{d}\tau \,\mathrm{d}s + k_2\int_{0}^{t}S_{\mathcal{G}}(t-s) |u(s)|^{p_2-1}u(s) \mathrm{d}s.
	\end{equation}
	Let us prove that  $\Lambda: B_{T}(R) \rightarrow B_{T}(R)$. Using Lemma \ref{Lp-Lqestimate}, we obtain for any $u \in B_{T}(R)$, 
	\begin{align*}
		\|\Lambda(u)(t)\|_{L^{r_i}} & \leq \|S_{\mathcal{G}}(t)u_0\|_{L^{r_i}} + \left\| \int_{0}^{t}S_{\mathcal{G}}(t-s) \int_0^s(s-\tau)^{-\gamma} \abs{u(\tau)}^{p_1-1}u(\tau) \, \mathrm{d}\tau \,\mathrm{d}s \right\|_{L^{r_i}}  \\
		& \quad +\left\|\int_{0}^{t}S_{\mathcal{G}}(t-s) |u(s)|^{p_2-1}u(s) \mathrm{d}s\right\|_{L^{r_i}} \\[.1cm]
		&\leq C\|u_0\|_{L^q}t^{-\alpha_i} +C   \int_{0}^{t} (t-s)^{-\alpha_i} \int_0^s (s-\tau)^{-\gamma} \|u(\tau)\|^{p_1}_{L^{r_1}} \, \mathrm{d}\tau \,\mathrm{d}s  \\
		& \quad+C \int_{0}^{t}(t-s)^{-\alpha_i}  \|u(s)\|^{p_2}_{L^{r_2}} \mathrm{d}s \\
		&\leq C\|u_0\|_{L^q}t^{-\alpha_i} + C\left(\sup_{t \in (0,T)}t^{\alpha_1} \|u(t)\|_{L^{r_1}}\right)^{p_1}   \int_{0}^{t} (t-s)^{-\alpha_i} \int_0^s (s-\tau)^{-\gamma} \tau^{-\alpha_1p_1} \, \mathrm{d}\tau \,\mathrm{d}s  \\
		& \quad+ C\left(\sup_{t \in (0,T)}t^{\alpha_2} \|u(t)\|_{L^{r_2}}\right)^{p_2} \int_{0}^{t}(t-s)^{-\alpha_i} s^{-\alpha_2 p_2}\mathrm{d}s \\
		&\leq C\|u_0\|_{L^q}t^{-\alpha_i} + C \|u\|^{p_1}_{X_T}  t^{-\alpha_i-\alpha_1 p_1-\gamma+2} + C \|u\|^{p_2}_{X_T}  t^{-\alpha_i-\alpha_2 p_2+1} ,
	\end{align*}
	for all $t\in (0,T)$, $i=1,2$, where we have used the fact  that $\gamma<1$ and that assumption $q>\max\limits_{i=1,2}\frac{N+2k}{2}(p_i-1)$ implies $\alpha_i p_i<1$ and $\gamma < 2-\alpha_1 p_1$. It follows
	\begin{equation*}
		t^{\alpha_i} \|\Lambda(u)(t)\|_{L^{r_i}} \leq R + C\left( R^{p_1} t^{-\alpha_1 p_1-\gamma+2} +  R^{p_2} t^{-\alpha_2 p_2+1}  \right),
	\end{equation*}
	for all $t\in(0,T)$, and so 
	\begin{equation}\label{defined}
		\max_{i=1,2} \sup_{t \in (0,T)}t^{\alpha_i}\|\Lambda(u)(t)\|_{L^{r_i}} \leq R +C\left(  R^{p_1} T^{-\alpha_1 p_1-\gamma+2} +  R^{p_2} T^{-\alpha_2 p_2+1}  \right) \leq 2R, 
	\end{equation}
	for a sufficiently small $T>0$ (depending on $R$). Moreover, 
	\begin{align*}
		\|\Lambda(u)(t)\|_{L^{q}} & \leq  C\|u_0\|_{L^q} +C   \int_{0}^{t} \int_0^s (s-\tau)^{-\gamma} \|u(\tau)\|^{p_1}_{L^{r_1}} \, \mathrm{d}\tau \,\mathrm{d}s +C \int_{0}^{t}  \|u(s)\|^{p_2}_{L^{r_2}} \mathrm{d}s \\
		&\leq C\|u_0\|_{L^q}+ C \|u\|^{p_1}_{X_T}  t^{-\alpha_1 p_1-\gamma+2} + C \|u\|^{p_2}_{X_T}  t^{-\alpha_2 p_2+1} ,
	\end{align*}
	for all $t\in (0,T)$, and so 
	\begin{equation}\label{newdefined}
		\sup_{t \in (0,T)}\|\Lambda(u)(t)\|_{L^{q}} \leq R +C\left(  R^{p_1} T^{-\alpha_1 p_1-\gamma+2} +  R^{p_2} T^{-\alpha_2 p_2+1}  \right) \leq 2R.
	\end{equation}	
	By combining \eqref{defined}-\eqref{newdefined}, we conclude that $\|\Lambda(u)\|_{X_T}\leq 4 R$ i.e. $\Lambda(u) \in B_{T}(R)$.\\
	Similarly, one can see that $\Lambda$ is a contraction. For $u,v \in B_T(R)$, we have
	\begin{align*}
		&t^{\alpha_i}\|\Lambda(u)(t)-\Lambda(v)(t)\|_{L^{r_i}} \\
		& \leq  C t^{\alpha_i} R^{p_1-1} \int_{0}^{t} (t-s)^{-\alpha_i} \int_0^s (s-\tau)^{-\gamma} \tau^{-\alpha_1(p_1-1)} \|u(\tau)-v(\tau)\|_{L^{r_1}} \mathrm{d}\tau \mathrm{d}s \\
		& \quad+ C t^{\alpha_i} R^{p_2-1} \int_{0}^{t} (t-s)^{-\alpha_i}s^{-\alpha_2(p_2-1)} \|u(s)-v(s)\|_{L^{r_2}} \mathrm{d}s  \\
		&\leq  C \left( R^{p_1-1} t^{2-\gamma-\alpha_1p_1} + R^{p_2-1} t^{1-\alpha_2p_2} \right) \left(\max_{i=1,2} \sup_{t \in (0,T)}t^{\alpha_i} \|u(t) - v(t)\|_{L^{r_i}}\right) ,
	\end{align*}
	for all $t\in(0,T)$, $i=1,2$. From the choices for the parameter $p_i,q,\gamma$, a sufficiently small $T>0$, we have
	\begin{equation*}
		d(\Lambda(u),\Lambda(v)) \leq \frac{1}{2}d(u,v).
	\end{equation*}
	Consequently, by the contraction principle, $\Lambda$  has a unique fixed point $u$ in $B_T(R)$.\\

	\noindent {\it Step 2. Regularity.} One can easily check that $f\in L^1((0,T),L^q(\R^{N+k}))$, with
	$$f(t):=k_1 \int_0^t(t-s)^{-\gamma}\abs{u(s)}^{p_1-1}u(s) \, \mathrm{d}s  + k_2 |u(t)|^{p_2-1}u(t),\qquad \hbox{for all}\,\,t\in (0,T).$$
	Applying Lemma 4.1.5 in \cite{Caz-Hau-98}, using the continuity of the semigroup $S_{\mathcal{G}}(t)$,  we conclude that $u\in C([0,T],L^q(\R^{N+k}))$.\\
	
	\noindent {\it Step 3. Uniqueness in $X_T$.} For sufficiently small $\tilde{T}>0$, \eqref{conv} implies that $u,v\in B_{\tilde{T}}(R)$. Thus $u\equiv v$ on $[0,\tilde{T}]$. For $t\in[\tilde{T},T]$, perform as in the contraction part of $\Lambda$, we get
		\begin{align*}
			t^{\alpha_i}\|u(t)-v(t)\|_{L^{r_i}} 
			\leq & C t^{\alpha_i} \|u\|_{X_T}^{p_1-1} \int_{\tilde{T}}^{t} (t-s)^{-\alpha_i} s^{1-\gamma-\alpha_1 p_1} \left(s^{\alpha_1} \|u(s) - v(s)\|_{L^{r_1}}\right) \mathrm{d}s \\
			& + C t^{\alpha_i} \|u\|_{X_T}^{p_2-1} \int_{\tilde{T}}^{t} (t-s)^{-\alpha_i}s^{-\alpha_2 p_2} \left(s^{\alpha_2} \|u(s) - v(s)\|_{L^{r_2}}\right) \mathrm{d}s    .
		\end{align*}
		Define
		$$\xi(t) = t^{\alpha_1} \|u(t) - v(t)\|_{L^{r_1}} +t^{\alpha_2} \|u(t) - v(t)\|_{L^{r_2}},$$
		and
		$$M=C (T^{\alpha_1}+T^{\alpha_2}) \left(  \|u\|_{X_T}^{p_1-1} + \|u\|_{X_T}^{p_2-1}\right) \left(\tilde{T}^{-\alpha_2 p_2} + T^{1-\gamma}\tilde{T}^{-\alpha_1 p_1}\right) $$
		Then
		$$\xi(t) \leq M  \int_{0}^{t} ((t-s)^{-\alpha_1}+(t-s)^{-\alpha_2})  \xi(s) \mathrm{d}s .$$
		Then, the singular Gronwall inequality (see e.g \cite[Lemma~3.3]{Amann-84}) gives us $\xi(t)=0$, whence we proved the uniqueness in $X_T$.

	\noindent {\it Step 4. Uniqueness in $C([0,T];L^q(\R^{N+k}))$.} Let $u$ and $v$ be two $L^q$-mild solutions of \eqref{*} with $u(0)=v(0)=u_0$, and such that $q\geq \max\limits_{i=1,2}  p_i$. We have
	\begin{eqnarray*}
		u(t)-v(t)&=&k_1 \int_{0}^{t}S_{\mathcal{G}}(t-s) \int_0^s(s-\tau)^{-\gamma}\left(\abs{u(\tau)}^{p_1-1}u(\tau)-\abs{v(\tau)}^{p_1-1}v(\tau)\right) \, \mathrm{d}\tau \,\mathrm{d}s\\
		&{}&\quad +\, k_2\int_{0}^{t}S_{\mathcal{G}}(t-s)\left( |u(s)|^{p_2-1}u(s)-|v(s)|^{p_2-1}v(s) \right)\mathrm{d}s.
	\end{eqnarray*}
	for all $t\in[0,T]$.  Using the $L^{q/p_i}-L^q$ estimate (Lemma \ref{Lp-Lqestimate}), we get
	\begin{eqnarray*}
		&&\|u(t)-v(t)\|_{L^q} \\
		&&\leq C \int_{0}^{t}(t-s)^{-\alpha_1p_1} \int_0^s(s-\tau)^{-\gamma}\|\left(\abs{u(\tau)}^{p_1-1}+\abs{v(\tau)}^{p_1-1}\right)|u(\tau)-v(\tau)|\|_{L^{q/p_1}} \, \mathrm{d}\tau \,\mathrm{d}s\\
		&{}&\quad +\, C \int_{0}^{t}(t-s)^{-\alpha_2p_2} \|\left(\abs{u(s)}^{p_2-1}+\abs{v(s)}^{p_2-1}\right)|u(s)-v(s)|\|_{L^{q/p_2}}  \,\mathrm{d}s\\
		&&\leq C \int_{0}^{t}(t-s)^{-\alpha_1p_1} \int_0^s(s-\tau)^{-\gamma}\left(\|u(\tau)\|_{L^q}^{p_1-1}+\|v(\tau)\|_{L^q}^{p_1-1}\right)\|u(\tau)-v(\tau)\|_{L^q} \, \mathrm{d}\tau \,\mathrm{d}s\\
		&{}&\quad +\, C \int_{0}^{t}(t-s)^{-\alpha_2p_2} \left(\|u(s)\|_{L^q}^{p_2-1}+v(s)\|_{L^q}^{p_2-1}\right)\|u(s)-v(s)\|_{L^q}  \,\mathrm{d}s\
	\end{eqnarray*}
	where we have used H\"older's inequality with $\frac{p_i}{q}=\frac{1}{q}+\frac{p_i-1}{q}$, $i=1,2$. Define
	$$M=\sup_{t \in [0,T]} \|u(t)\|_{L^q}+\|v(t)\|_{L^q}  \qquad\hbox{and}\qquad \xi(t) =  \sup_{s \in [0,t]}  \|u(s) - v(s)\|_{L^q},$$
	for all $t\in[0,T]$. We conclude that
	\begin{eqnarray*}
		\xi(t) &\leq& C\,M^{p_1-1}  \int_{0}^{t}(t-s)^{-\alpha_1p_1}s^{1-\gamma}\xi(s) \,\mathrm{d}s +\, C\,M^{p_2-1} \int_{0}^{t}(t-s)^{-\alpha_2p_2} \xi(s) \,\mathrm{d}s\\
		&\leq& C\left(M^{p_1-1}T^{2-\gamma-\alpha_1p_1}+M^{p_2-1}T^{1-\alpha_2p_2}\right)  \xi(t) ,
	\end{eqnarray*}
	for all $t\in[0,T]$, which implies that $\xi\equiv0$ on $[0,\widetilde{T}]$ for sufficiently small $\widetilde{T}$. By a standard continuation argument, one can extend the uniqueness to the entire interval $[0,T]$. \\
	
	\noindent {\it Step 5. Maximal solution and blow-up alternative.} Using the uniqueness of the $L^q$-mild solution, we conclude the existence
	of a solution on a maximal interval $[0,T_{\max})$ where
	\[
	T_{\max}:=\sup\left\{T>0\;;\;\text{there exist a mild solution $u\in C([0,T];L^q(\R^{N+k}))$
		to \eqref{*}}\right\}.
	\]
	Obviously, $T_{\max}\leq \infty$. To prove that $\|u\|_{L^\infty((0,t),L^q(\mathbb{R}^{N+k}))}\rightarrow\infty$ as $t\rightarrow T_{\max},$
	whenever $T_{\max}<\infty,$ we proceed by contradiction. If
	$$\liminf_{t\rightarrow T_{\max}}\|u\|_{L^\infty((0,t),L^q(\mathbb{R}^{N+k}))}=:L<\infty,$$
	then there exists a time sequence $\{t_m\}_{m\geq0}$ tending to $T_{\max}$ as $m\rightarrow\infty$ and such that
	$$\sup_{m\in\mathbb{N}}\sup_{0<s\leq t_m}\|u(s)\|_{L^q}\leq L+1.$$
	Moreover, if $0\leq t_m\leq t_m+\tau< T_{\max},$ using \eqref{branda}, we can write
	\begin{eqnarray}\label{newIE+}
		u(t_m+\tau)&=&S_{\mathcal{G}}(\tau)u(t_m)+k_1\int_0^\tau S_{\mathcal{G}}(\tau-s)\int_0^s(s-\sigma)^{-\gamma}\abs{u}^{p_1-1}u(t_m+\sigma)\,d\sigma\,ds\nonumber\\
		&& +\,k_1\int_0^\tau S_{\mathcal{G}}(\tau-s)\int_0^{t_m}(t_m+s-\sigma)^{-\gamma}\abs{u}^{p_1-1}u(\sigma)\,d\sigma\,ds\nonumber\\
		&& +\,k_2\int_0^\tau S_{\mathcal{G}}(\tau-s)\abs{u}^{p_2-1}u(s+t_m)\,ds.
	\end{eqnarray}
	Using the fact that the third term on the right-hand side in $(\ref{newIE+})$ depends only on the values of $u$ in the interval
	$(0,t_m)$ and using again a fixed-point argument with $u(t_m)$ as initial condition, one can deduce that there exists $T(L + 1) > 0$, depends on $L+1$, such that the solution
	$u(t)$ can be extended on the interval $[t_m, t_m + T(L + 1)]$ for any $m\geq0$. Thus, by the definition of the maximality time, $T_{\max}\geq t_m+T(L+1)$, for any $m\geq0$. We get the desired contradiction by letting $m\rightarrow\infty$.\\
	
	\noindent {\it Step 6. Nonnegativity of solutions.} Assume $u_0\geq0$ and $k_1,k_2\geq0$.  In this case, we can construct a nonnegative solution on some interval $[0,T]$ by applying the fixed point argument within the positive cone $X_T^+=\{u\in X_T;\;u\geq0\}$, and using Proposition \ref{properties}-(vi). By the uniqueness of solutions, it then follows that $u(t)\geq0$ for all $t\in(0,T_{\max}).$\\
	
	\noindent {\it Step 7. Decay of solution. }As $u\in B_T(R)$, by choosing $R=C\|u_0\|_{L^{q}}$, we have
	\begin{equation*}
		t^{\alpha_i} \|u(t)\|_{L^{r_i}} \leq  Ct^{\alpha_i}  \|S_\mathcal{G}(t)u_0\|_{L^{r_i}} + C R^{p_1}   t^{-\alpha_1 p_1-\gamma+2} + C R^{p_2}  t^{-\alpha_2 p_2+1}  \longrightarrow0,
	\end{equation*}
	when $t$ goes to zero, where we have used Lemma \ref{lemmaSg} and  the assumptions on the parameters $p_i,q,\gamma$.\\
	
	\noindent {\it Step 8. Continuous dependence.} 	Let $u,v\in X_T$ be two mild solutions starting at $u_0$ and $v_0$, respectively. The existence proof shows that $u,v\in B_T(R)$ for any $u_0,v_0\in L^q(\R^{N+k}))$ satisfying $C\max\{\|u_0\|_{L^q},\|v_0\|_{L^q}\}\leq R$. Hence,
	\begin{eqnarray*}
		&& t^{\alpha_i}\|u(t)-v(t)\|_{L^{r_i}} \\
		&& \leq  C\,\|u_0-v_0\|_{L^q} + C\,t^{\alpha_i}\int_0^t(t-s)^{-\alpha_i} \int_0^s (s-\tau)^{-\gamma} \|u(\tau)\|_{L^{r_1}}^{p_1-1}\|u(\tau) - v(\tau) \|_{L^{r_1}} \mathrm{d}\tau \mathrm{d}s \nonumber \\
		&{}&\quad +\, 
		C\,t^{\alpha_i}\int_0^t(t-s)^{-\alpha_i} \|u(s)\|_{L^{r_2}}^{p_2-1}\|u(s) - v(s) \|_{L^{r_2}} \mathrm{d}s  \nonumber\\
		&&\leq  C\,\|u_0-v_0\|_{L^q} + C\left( R^{p_1-1}T^{2-\gamma-\alpha_1 p_1} + R^{p_2-1}T^{1-\alpha_2 p_2} \right) \max_{i=1,2} \sup_{t \in (0,T)}t^{\alpha_i}\|u(t)-v(t)\|_{L^{r_i}}\nonumber\\
		&&\leq  C\,\|u_0-v_0\|_{L^q} + \frac{1}{2}\max_{i=1,2} \sup_{t \in (0,T)}t^{\alpha_i}\|u(t)-v(t)\|_{L^{r_i}},
	\end{eqnarray*}
	for all $t\in(0,T)$, $i=1,2$, by using the smallness choice of $T$. This implies
	\begin{equation}\label{deplin1}
		\max_{i=1,2} \sup_{t \in (0,T)}t^{\alpha_i}\|u(t)-v(t)\|_{L^{r_i}} \leq C\|u_0-v_0\|_{L^q}.
	\end{equation}
	In addition,
	\begin{eqnarray*}
		&&\|u(t)-v(t)\|_{L^q} \\
		&&\leq  C\,\|u_0-v_0\|_{L^q} + C\int_0^t \int_0^s (s-\tau)^{-\gamma} \|u(\tau)\|_{L^{r_1}}^{p_1-1}\|u(\tau) - v(\tau) \|_{L^{r_1}} \mathrm{d}\tau \mathrm{d}s  \\
		&{}&\quad + \,C\int_0^t \|u(s)\|_{L^{r_2}}^{p_2-1}\|u(s) - v(s) \|_{L^{r_2}} \mathrm{d}s  \\
		&&\leq C\,\|u_0-v_0\|_{L^q} + C\left( R^{p_1-1}T^{2-\gamma-\alpha_1 p_1} + R^{p_2-1}T^{1-\alpha_2 p_2} \right) \max_{i=1,2} \sup_{t \in (0,T)}t^{\alpha_i}\|u(t)-v(t)\|_{L^{r_i}}\nonumber\\
		&&\leq  C\,\|u_0-v_0\|_{L^q} ,
	\end{eqnarray*}
	for all $t\in[0,T]$, where we have used \eqref{deplin1} and the smallness of $T$. It follows
	\begin{equation}\label{deplin2}
		\sup_{t \in [0,T]} \|u(t)-v(t)\|_{L^q} \leq C\|u_0 - v_0\|_{L^q} .
	\end{equation}
	Combining \eqref{deplin1}-\eqref{deplin2} we get our desired estimate \eqref{contdep}.
\end{proof}

\subsection{Local existence with smooth initial condition}

The results in Theorem \ref{teoexun} hold for $u_0\in C_0(\R^{N+k})$. In fact, in this case the proof would run easier (see e.g. \cite{CDW, Fino-Ki-12}).
\begin{theorem}[Local existence with smooth initial condition]\label{T0+}${}$\\
	Let $u_0\in C_0(\mathbb{R}^{N+k})$ and $p_i>1$, for $i=1,2$. Then, there exist a maximal
	time $0<T_{\max}=T_{\max}(u_0)\leq\infty$ and a unique mild solution 
	$$u\in C([0,T_{\max}),C_0(\mathbb{R}^{N+k}))$$
	to problem \eqref{*}. Moreover, either $T_{\max}=\infty$, or else $T_{\max}<\infty$ and 
	$$\|u\|_{L^\infty((0,t)\times\mathbb{R}^{N+k})}\rightarrow\infty\qquad \mathrm{as}\,\, t\rightarrow T_{\max}.$$ In addition, if $u_0\in L^r(\mathbb{R}^{N+k})\cap C_0(\mathbb{R}^{N+k})$ for some $1\leq r<\infty$, then 
	$$u\in C([0,T_{\max}),L^r(\mathbb{R}^{N+k})\cap C_0(\mathbb{R}^{N+k})).$$
	Finally, if $u_0\geq0$ and $k_1,k_2\geq0$, then then the mild solution remains nonnegative for all $t\in[0,T_{\max})$.
\end{theorem}

\subsection{Comparison Principle}\label{comparisonp}

In this subsection, we establish a comparison principle for mild solutions of \eqref{*}, in which the nonlinear terms $u\mapsto |u|^{p_i-1}u$ (for $i=1,2$) are replaced by general nonlinearities $f$ and $g$ satisfying the following growth conditions: 
\begin{equation}\label{growth1}
	|f(u)-f(v)| \leq c|u-v|(|u|^{p_1-1} + |v|^{p_1-1}),\quad\quad |g(u)-g(v)| \leq c|u-v|(|u|^{p_2-1} + |v|^{p_2-1}).
\end{equation}
For a given an initial condition $u_0$, we denote by $u(t;u_0,f,g)$ the $X$-mild solution of \eqref{*}, where the notation highlights the dependence on the initial data and the nonlinear terms $f$ and $g$. Here, $X$ denotes either $C_0(\mathbb{R}^{N+k})$ or $L^q(\mathbb{R}^{N+k})$, with $1\leq q <\infty$, satisfying
$$q>\max\limits_{i=1,2}\frac{N+2k}{2}(p_i-1).$$

We are now in a position to state a comparison principle, motivated by the use of the Picard iteration associated with the mild solution. For this purpose, we introduce the notion of a mild supersolution, which will be used in the formulation of the principle. A function $v:[0,T]\to X$ is said to be a {\bf mild supersolution} of \eqref{*}, if $v(t) \geq \Lambda_{fg}(v)(t)$, for all $t\in[0,T]$, where the operator $\Lambda_{fg}$ is defined in equation \eqref{lambdafg} below.

\begin{theorem}\label{comparison}
	Let $u_0,v_0\in X$, $k_1,k_2\geq 0$, and $f,g, \tilde{f}, \tilde{g}$ functions satisfying \eqref{growth1}.
	\begin{itemize}
		\item[(i)] If $0\leq u_0\leq v_0$, and $f(\mu)\leq \tilde{f}(\nu)$, $g(\mu)\leq \tilde{g}(\nu)$, for all $0\leq\mu\leq \nu$, then 
		$$0\leq u(t;u_0,f,g) \leq v(t;v_0,\tilde{f},\tilde{g})\quad\text{for all}\,\, t\in[0,\tilde{T}],$$
		where $\tilde{T}>0$ is a common time of existence for $u$ and $v$.
		\item[(ii)]  If $v$ is a mild supersolution of \eqref{*}, and $u$ is a mild solution with $u_0\geq0$, then $0\leq u\leq v$ for all $t\in[0,T]$.
	\end{itemize}
\end{theorem}

\begin{proof} We begin by defining the Picard sequences:
	$$u_{1}(t) =S_\mathcal{G}(t)u_0,\qquad u_{k}(t) =\Lambda_{fg} (u_{k-1})(t),\,\,\, k\geq2 ,$$
	$$v_{1}(t) =S_\mathcal{G}(t)v_0 ,\qquad v_{k}(t) =\widetilde{\Lambda}_{\tilde{f}\tilde{g}} (v_{k-1})(t),\,\,\, k\geq2 , $$
	where $\Lambda_{fg}$ and $\widetilde{\Lambda}_{\tilde{f}\tilde{g}}$ refer to the operator introduced in equation \eqref{solop}, emphasizing the dependence on the nonlinearities in the right-hand side:
	\begin{equation}\label{lambdafg}
		\Lambda_{fg}(u)(t):=S_{\mathcal{G}}(t)u_0 +k_1 \int_{0}^{t}S_{\mathcal{G}}(t-s) \int_0^s(s-\tau)^{-\gamma}f(u(\tau))\, \mathrm{d}\tau \,\mathrm{d}s +k_2 \int_{0}^{t}S_{\mathcal{G}}(t-s)g(u(s)) \mathrm{d}s,
	\end{equation}
	$$\widetilde{\Lambda}_{\tilde{f}\tilde{g}}(u)(t):=S_{\mathcal{G}}(t)v_0 +k_1 \int_{0}^{t}S_{\mathcal{G}}(t-s) \int_0^s(s-\tau)^{-\gamma}\tilde{f}(u(\tau))\, \mathrm{d}\tau \,\mathrm{d}s +k_2 \int_{0}^{t}S_{\mathcal{G}}(t-s)\tilde{g}(u(s)) \mathrm{d}s.$$
	By Proposition \ref{properties}-(vi), we know that the semigroup $S_\mathcal{G}(t)$ preserves order:
	\begin{equation}\label{monotonicity}
		S_\mathcal{G}( t)\varphi \leq S_\mathcal{G}(t)\psi,\quad \hbox{whenever}\,\,\, \varphi\leq \psi.
	\end{equation} 
	Hence, $0\leq u_1(t)\leq v_1(t)$. We now apply induction to show that $0\leq u_k(t)\leq v_k(t)$, for all $k\in\N$. Assume this holds for some fixed $k$. Then, using the monotonicity of the semigroup $S_\mathcal{G}(t)$ (cf. \eqref{monotonicity}) together with the structural conditions imposed on the nonlinearities $f, \tilde{f},g,\tilde{g}$, we obtain
	$$0\leq u_{k+1}(t)=\Lambda_{fg}(u_k)(t)\leq \widetilde{\Lambda}_{\tilde{f}\tilde{g}}(v_k)(t)=v_{k+1}(t).$$
	Hence, the inequality is preserved at each iteration step. This proves part $(i)$, noting that the corresponding mild solutions are obtained as pointwise limits of the Picard sequences, depending on the functional setting:
	\begin{itemize}
		\item  If $X=C_0(\mathbb{R}^{N+k})$, then the convergence is uniform on compact sets and, in particular, pointwise. Thus,
		$$u(t;u_0,f,g) = \lim\limits_{n\rightarrow\infty} u_n(t),\qquad  v(t;v_0,\tilde{f},\tilde{g}) = \lim\limits_{n\rightarrow\infty} v_n(t),$$
		where the limits are taken pointwise.
		\item  If $X=L^q(\mathbb{R}^{N+k})$,  then by norm convergence and standard results, there exists a common subsequence along which both sequences converge pointwise almost everywhere. That is,
		$$u(t;u_0,f,g) = \lim\limits_{k\rightarrow\infty} u_{n_k}(t),\qquad  v(t;v_0,\tilde{f},\tilde{g}) = \lim\limits_{k\rightarrow\infty} v_{n_k}(t),$$ 
		where the limits are taken almost everywhere along the same subsequence $(n_k)_k$.\\
	\end{itemize}
	To prove part $(ii)$, it suffices to observe that since $v$ is a mild supersolution, we can construct a non-increasing Picard sequence $(w_k)_k$ defined by
	$$
	w_{1}(t) =v(t),\qquad\quad w_{k}(t) =\Lambda_{fg} (w_{k-1})(t),\,\,\hbox{for}\,\, k\geq2.
	$$
	By part $(i)$, we have
	$$w_{2}(t) =\Lambda_{fg} (w_1)(t)=\Lambda_{fg} (v)(t)\leq v(t)=w_1(t),$$ 
	and an induction argument using the monotonicity property of $\Lambda_{fg}$ (as established in part $(i)$) yields that the sequence $(w_k)_k$  is non-increasing. Moreover, since the nonlinearities $f,g$ are nonnegative and the semigroup $S_\mathcal{G}(t)$ preserves nonnegativity, each $w_k(t)$ remains bounded, satisfying $0\leq w_k(t)\leq v(t)$ for all $k\geq1$.
	\begin{itemize}
		\item  If the functional space is $X=C_0(\mathbb{R}^{N+k})$, then $(w_k)_k$ converges uniform on compact sets and, in particular,
		pointwise to a limit function $w(t)$. By standard monotone convergence results in Banach spaces, the convergence is strong in $C([0,T],C_0(\mathbb{R}^{N+k})))$. Passing to the limit in the iteration yields
		$$w(t)= \Lambda_{fg} (w)(t)\leq v(t),$$
		in particular, $w$ is a mild solution of problem \eqref{*}. By uniqueness, we conclude that $u=w\leq v$.
		\item  If instead $X=L^q(\mathbb{R}^{N+k})$,  then the sequence $(w_k)_k$ converges in norm in $L^q$, and in particular, there exists a subsequence $(w_{k_j})_j$ that converges almost everywhere to the same limit function $w$. Passing to the limit almost everywhere in the integral formulation gives
		$$w(t)= \Lambda_{fg} (w)(t)\leq v(t),$$
		so again $w$ is a mild solution of \eqref{*}. Uniqueness implies $u=w\leq v$.
	\end{itemize} 
\end{proof}

\begin{remark}
	In particular, Theorem \ref{comparison} gives another proof for the nonnegativity in Theorems \ref{teoexun} and \ref{T0+}.
\end{remark}

\section{Global Existence}\label{global}

In this section, we prove the existence of global solutions. We also prove some comparison principles needed in our approach.

Let us define the following critical exponents.
\begin{equation}\label{critexp}
	p_1^* = \max\left\{\frac{1}{\gamma} , p_\gamma\right\} \qquad\hbox{and}\qquad	p_2^* = 1+ \frac{2}{N+2k} ,
\end{equation}
where
\begin{eqnarray*}
	p_\gamma =	1+ \frac{4-2\gamma}{N+2k-2+2\gamma} ,
\end{eqnarray*}
and
\begin{equation}
	p_2^{**}=\max\left\{\frac{\gamma-\gamma^2+1}{\gamma(2-\gamma)} , 1+\frac{2}{N+2k-2+2\gamma}\right\} .
\end{equation}

We also define $\tilde{p}_2:= (p_1+1-\gamma)/(2-\gamma)$ and the scaled exponents as
$$q_{\mathrm{sc}}=\left\{\begin{array}{ll}
	\displaystyle \frac{(N+2k)(p_1-1)}{2(2-\gamma)}  =: q_{sc_1},&\qquad\hbox{if}\,\, p_2\geq \tilde{p}_2,\\\\
	\displaystyle \frac{(N+2k)(p_2-1)}{2} =: q_{sc_2} ,&\qquad\hbox{if}\,\, p_2\leq \tilde{p}_2.\\
\end{array}
\right.$$
\begin{theorem}\label{global1}
	The following assertions hold:
	\begin{enumerate}
		\item[$(\mathrm{i})$] Let $u_0\in  C_0(\R^{N+k})\cap L^{q_{\mathrm{sc}}}(\R^{N+k})$, with $\|u_0\|_{L^{q_{\mathrm{sc}}}}$ sufficiently small.  If $k_1,k_2\in\mathbb{R}$, $p_1>p_1^*$ and $p_2=\tilde{p}_2$, then there exists a unique global mild solution $u\in C([0,\infty);C_0(\R^{N+k}))$ of \eqref{*}.
		\item[$(\mathrm{ii})$] Let $u_0\in  C_0(\R^{N+k})\cap L^{q_{\mathrm{sc}}}(\R^{N+k})$, with $\|u_0\|_{L^{\infty}}$ and  $\|u_0\|_{L^{q_{\mathrm{sc}}}}$ sufficiently small. If $u_0\geq0$,  $k_1,k_2>0$, $p_1>p_1^{*}$ and $p_2>p_2^{**}$, then problem \eqref{*} admits a unique global mild solution 
		$u\in C([0,\infty);C_0(\mathbb{R}^{N+k}).$
		\item[$(\mathrm{iii})$] Let $u_0\in X$ where $X$ is the function space defined in Subsection \ref{comparisonp}. If $u_0\geq0$, $k_i\leq0$ for $i=1,2$, then any nonnegative mild solution of \eqref{*} is global.
		\item[$(\mathrm{iv})$] Let $u_0\in  C_0(\R^{N+k})\cap L^{q_{\mathrm{sc_1}}}(\R^{N+k})$, with $\|u_0\|_{L^{q_{\mathrm{sc_1}}}}$ sufficiently small. If $u_0\geq0$, $k_1>0$, $k_2\leq0$, and $p_1>p_1^*$, then any nonnegative mild solution of \eqref{*} is global.
		\item[$(\mathrm{v})$]  Let $u_0\in  C_0(\R^{N+k})\cap L^{q_{\mathrm{sc_2}}}(\R^{N+k})$, with $\|u_0\|_{L^{q_{\mathrm{sc_2}}}}$ sufficiently small. If $u_0\geq0$, $k_2>0$, $k_1\leq 0$, and $p_2>p_2^*$, then any nonnegative mild solution of \eqref{*} is global.
	\end{enumerate}
\end{theorem}

\begin{proof} We divide the proof into five separate cases, each corresponding to one of the scenarios described in the theorem. These cases differ depending on the signs and values of the coefficients $k_1,k_2$ and the exponents $p_1,p_2$, as well as assumptions on the initial data.

	\noindent {\bf Case $(\mathrm{i})$.} The proof is divided into two steps.\\
	{\it Step 1. Existence.} As $p_1>p_1^*$, then we have the possibility to
	take a positive constant $q>0$ so that
	\begin{equation}\label{estiA}
		\frac{2-\gamma}{p_1-1}-\frac{1}{p_1}<\frac{N+2k}{2
			q}<\frac{1}{p_1-1},\quad q\geq p_1.
	\end{equation}
	It follows that
	\begin{equation}\label{estiB}
		q>\frac{(N+2k)(p_1-1)}{2}>q_{sc}>1.
	\end{equation}
	Let
	\begin{equation}\label{estiC}
		\beta:=\frac{N+2k}{2 q_{sc}}-\frac{N+2k}{2
			q}=\frac{2-\gamma}{p_1-1}-\frac{N+2k}{2
			q}=\frac{1}{p_2-1}-\frac{N+2k}{2
			q}.
	\end{equation}
	Thus, by utilizing relations (\ref{estiA})-(\ref{estiC}) along with the condition $p_2<p_1$, one can easily verify that $p_i\beta<1$ for $i=1,2$, $q> p_2$, as well as the relations
	\begin{equation}\label{estiD}
		\beta>\frac{1-\gamma}{p_1-1}>0,\quad \frac{(N+2k)(p_1-1)}{2
			q}+(p_1-1)\beta+\gamma=2,\quad \frac{(N+2k)(p_2-1)}{2
			q}+(p_2-1)\beta=1,
	\end{equation}
	hold. As $u_0\in L^{q_{\mathrm{sc}}}(\R^{N+k})$, using \eqref{semigroup}, $q>q_{\mathrm{sc}}$, and (\ref{estiC}), we
	get
	\begin{equation}\label{estiE}
		\sup_{t>0}t^\beta\|S_{\mathcal{G}}(t)u_0\|_{L^q}\leq
		C \|u_0\|_{L^{q_{\mathrm{sc}}}}=\eta<\infty.
	\end{equation}
	Set
	\begin{equation}\label{estiF}
		\Xi:=\left\{u\in
		L^\infty((0,\infty),L^q(\mathbb{R}^{N+k}));\;\sup_{t>0}t^\beta\|u(t)\|_{L^q}\leq\delta\right\},
	\end{equation}
	where $\delta>0$ is to be chosen sufficiently small. If we define
	\begin{equation}\label{estiG}
		d_{\Xi}(u,v):=\sup_{t>0}t^\beta\|u(t)-v(t)\|_{L^q},\qquad \hbox{for all}\,\, u,v\in\Xi,
	\end{equation}
	then $(\Xi,d)$ is a nonempty complete metric space. For $u\in\Xi$, we define $\Phi(u)$ by
	\begin{equation}\label{estiH}
		\Phi(u)(t):=S_{\mathcal{G}}(t)u_0 +k_1 \int_{0}^{t}S_{\mathcal{G}}(t-s) \int_0^s(s-\tau)^{-\gamma}\abs{u(\tau)}^{p_1-1}u(\tau) \, \mathrm{d}\tau \,\mathrm{d}s + k_2\int_{0}^{t}S_{\mathcal{G}}(t-s) |u(s)|^{p_2-1}u(s) \mathrm{d}s,
	\end{equation}
	for all $t\geq0$. Let us prove that  $\Phi: \Xi \rightarrow \Xi$. Using $q\geq p_i$ for $i=1,2$, \eqref{estiE}, \eqref{estiF}, and Lemma \ref{Lp-Lqestimate}, we obtain for any $u \in \Xi$, 
	\begin{eqnarray}\label{estiI}
		t^\beta\|\Phi(u)(t)\|_{L^q}&\leq& \eta+\,Ct^\beta\int_0^t(t-s)^{-\frac{(N+2k)(p_1-1)}{2q}}
		\int_0^s(s-\sigma)^{-\gamma}\|u(\sigma)\|^{p_1}_{L^{q}}\,d\sigma\,ds\nonumber\\
		&{}&+\,Ct^\beta\int_0^t(t-s)^{-\frac{(N+2k)(p_2-1)}{2q}}\|u(s)\|^{p_2}_{L^{q}}\,ds\nonumber\\
		&\leq&\eta+C\delta^{p_1} t^\beta\int_0^t\int_0^s(t-s)^{-\frac{(N+2k)(p_1-1)}{2
				q}}(s-\sigma)^{-\gamma}\sigma^{-\beta p_1}\,d\sigma\,ds\nonumber\\
		&{}&+\,C\delta^{p_2} t^\beta\int_0^t(t-s)^{-\frac{(N+2k)(p_2-1)}{2
				q}}s^{-\beta p_2}\,ds.
	\end{eqnarray}
	Next, using \eqref{estiA}, \eqref{estiD}, $p_2<p_1$, and $p_i\beta<1$ for $i=1,2$, we get
	\begin{equation}\label{estiJ}
		\int_0^t\int_0^s(t-s)^{-\frac{(N+2k)(p_1-1)}{2
				q}}(s-\sigma)^{-\gamma}\sigma^{-\beta p_1}\,d\sigma\,ds= C\int_0^t(t-s)^{-\frac{(N+2k)(p_1-1)}{\beta
				q}}s^{1-\gamma-\beta p_1}\,ds=C t^{-\beta},
	\end{equation}
	and
	\begin{equation}\label{estiJ+}
		\int_0^t(t-s)^{-\frac{(N+2k)(p_2-1)}{2
				q}}s^{-\beta p_2}\,ds=Ct^{1-\frac{(N+2k)(p_2-1)}{2
				q}-\beta p_2}=C t^{-\beta},
	\end{equation}
	for all $t\geq0.$ So, we deduce from \eqref{estiI}-\eqref{estiJ+}
	that
	\begin{equation}\label{estiK}
		t^\beta\|\Phi(u)(t)\|_{L^q}\leq \eta+C(\delta^{p_1} + \delta^{p_2}).
	\end{equation}
	Therefore, if $\eta$ and $\delta$ are chosen small enough so that $\eta+C(\delta^{p_1} + \delta^{p_2})\leq\delta,$ we see that
	$\Phi:\Xi\rightarrow\Xi.$ Similar calculations show that (assuming
	$\eta$ and $\delta$ small enough) $\Phi$ is a strict contraction, so
	it has a unique fixed point $u\in\Xi$ which is a mild solution of \eqref{*}.
	
	We now establish that $u\in C([0,\infty),C_0(\mathbb{R}^{N+k}))$. To begin, we first show that $u\in C([0,T],C_0(\mathbb{R}^{N+k}))$ for some sufficiently small $T>0$. Indeed, the previous argument ensures uniqueness in $\Xi_T,$ where for any $T>0,$
	$$
	\Xi_T:=\left\{u\in
	L^\infty((0,T),L^q(\mathbb{R}^{N+k}));\;\sup_{0<t<T}t^\beta\|u(t)\|_{L^q}\leq\delta\right\}.
	$$
	Let $\tilde{u}$ be the local solution of \eqref{*}
	constructed in Theorem \ref{T0+}. By applying \eqref{estiB}, we have $q_{sc}<q<\infty$, which implies 
	$$u_0\in C_0(\mathbb{R}^{N+k})\cap L^{q_{sc}}(\mathbb{R}^{N+k})\subset C_0(\mathbb{R}^{N+k})\cap L^{q}(\mathbb{R}^{N+k}).$$ 
	Thus,  Theorem \ref{T0+} guarantees that $\tilde{u}\in C([0,T_{\max}),C_0(\mathbb{R}^{N+k})\cap L^q(\mathbb{R}^{N+k}))$. Consequently, we have
	$$\sup\limits_{t\in(0,T)}t^\beta\|\tilde{u}(t)\|_{L^q}\leq\delta$$
	for sufficiently small  $T>0$. By the uniqueness of solutions, it follows that $u=\tilde{u}$ on
	$[0,T]$, leading to the conclusion that $u\in C([0,T],C_0(\mathbb{R}^{N+k}))$.
	
	Next, we show that $u\in C([T,\infty),C_0(\mathbb{R}^{N+k}))$ using a bootstrap argument. Indeed, for $t>T,$ we express $u(t)$ as
	\begin{eqnarray*}
		u(t)-S_{\mathcal{G}}(t)u_0 &=&
		k_1 \int_0^tS_{\mathcal{G}}(t-s)\int_0^T
		(s-\sigma)^{-\gamma}\abs{u}^{p_1-1}u(\sigma)\;d\sigma\,ds\\
		&&+k_1\int_0^tS_{\mathcal{G}}(t-s)\int_T^s(s-\sigma)^{-\gamma}\abs{u}^{p_1-1}u(\sigma)\;d\sigma\,ds\\
		&& +k_2 \int_0^TS_{\mathcal{G}}(t-s)\abs{u}^{p_2-1}u(s)\,ds + k_2\int_T^tS_{\mathcal{G}}(t-s)\abs{u}^{p_2-1}u(s)\,ds \\
		&\equiv& I_1(t)+I_2(t)+I_3(t)+I_4(t).
	\end{eqnarray*}
	Since $u\in C([0,T],C_0(\mathbb{R}^{N+k})),$ it follows that $I_1,I_3\in
	C([T,\infty),C_0(\mathbb{R}^{N+k})).$ Moreover, using similar calculations used to
	construct the fixed point, and noting that $t^{-\beta}\leq
	T^{-\beta}<\infty$, we conclude that $I_1,I_3\in C([T,\infty),L^q(\mathbb{R}^{N+k}))$.
	Next, from \eqref{estiB}, we have $q>(N+2k)(p_1-1)/2$, which guarantees the existence of some $r\in(q,\infty]$ such that
	\begin{equation}\label{estiL}
		\frac{N+2k}{2}\left(\frac{p_1}{q}-\frac{1}{r}\right)<1.
	\end{equation}
	For $T<s<t$, since $u\in
	L^\infty((0,\infty),L^q(\mathbb{R}^{N+k})),$ we obtain
	$$|u|^{p_1-1}u\in L^\infty((T,s),L^{q/p_1}(\mathbb{R}^{N+k}))\quad\hbox{and}\quad |u|^{p_2-1}u\in L^\infty((T,t),L^{q/p_2}(\mathbb{R}^{N+k})).$$
	Thus, applying (\ref{semigroup}), (\ref{estiL}), and the fact that $p_2<p_1$, it follows that $I_2, I_4\in
	C([T,\infty),L^r(\mathbb{R}^{N+k}))$.
	Since the terms $S_{\mathcal{G}}(t)u_0$, $I_1$, and $I_3$ belong to
	$C([T,\infty),C_0(\mathbb{R}^{N+k}))\cap
	C([T,\infty),L^q(\mathbb{R}^{N+k})),$ we conclude that $u\in
	C([T,\infty),L^r(\mathbb{R}^{N+k}))$. We ultimately deduce, using a bootstrap argument, that $u\in
	C([T,\infty),C_0(\mathbb{R}^{N+k})).$ \\	
	\noindent{\it Step 2. Uniqueness.} Let $u,v \in C([0,\infty),C_0(\mathbb{R}^{N+k}))$ be two global mild solutions for \eqref{*}. Let $T>0$, we have
	\begin{eqnarray*}
		&&\|u(t)-v(t)\|_{\infty}\\
		&&\leq C
		\int_0^t\int_0^s(s-\sigma)^{-\gamma}
		\|u(\sigma)-v(\sigma)\|_{\infty}\,d\sigma\,ds+C
		\int_0^t
		\|u(s-v(s)\|_{\infty}\,ds\\
		&&=C
		\int_0^t\int_\sigma^t(s-\sigma)^{-\gamma}\|u(\sigma)-
		v(\sigma)\|_{\infty}\,ds\,d\sigma +C
		\int_0^t
		\|u(s-v(s)\|_{\infty}\,ds\\
		&&=C \int_0^t(t-\sigma)^{1-\gamma}\|u(\sigma)-
		v(\sigma)\|_{\infty}\,d\sigma+C
		\int_0^t
		\|u(s-v(s)\|_{\infty}\,ds\\
		&&\leq C\left(T^{1-\gamma}+1
		\right)  \int_0^t
		\|u(s-v(s)\|_{\infty}\,ds.
	\end{eqnarray*}
	By Gronwall's inequality, we conclude that $u(t)=v(t)$ for every $t\in[0,T]$. Since $T>0$ is arbitrary, this completes the proof of Case $(\mathrm{i})$.\\
	
	\noindent {\bf Case $(\mathrm{ii})$.} We address this case using the comparison principle from Theorem \ref{comparison}. Two subcases arise based on the relative size of the exponent $p_2$.

	\underline{The case $p_2> \tilde{p}_2$.}  We aim to show that the mild solution $u$ of \eqref{*} satisfies 
	\begin{equation}\label{subsol}
		0\leq u\leq v,
	\end{equation} 
	where $v$ is the mild solution of the same equation with the exponent $p_2$ replaced with $\tilde{p}_2=(p_1+1-\gamma)/(2-\gamma)$. Since $p_1>p_1^{*}$, {\bf Case} (i) can be applied to ensure that the solution $v$ exists globally and is locally bounded. Then, if \eqref{subsol} holds, $u$ exists globally, by Theorem \ref{T0+}. Let us prove \eqref{subsol}. Indeed, for sufficiently small initial data $u_0\geq0$, with $0<\epsilon_0\ll1$, we have $0\leq u_0<\epsilon_0$. Let $\Phi$ and $\tilde{\Phi}$ denote the integral operators associated with the respective mild formulations for $u$ and $v$: 
	$$\Phi(u)(t):=S_{\mathcal{G}}(t)u_0 +k_1 \int_{0}^{t}S_{\mathcal{G}}(t-s) \int_0^s(s-\tau)^{-\gamma}(u(\tau))^{p_1}\, \mathrm{d}\tau \,\mathrm{d}s +k_2 \int_{0}^{t}S_{\mathcal{G}}(t-s)(u(s))^{p_2} \mathrm{d}s,$$
	$$\tilde{\Phi}(v)(t):=S_{\mathcal{G}}(t)u_0 +k_1 \int_{0}^{t}S_{\mathcal{G}}(t-s) \int_0^s(s-\tau)^{-\gamma}(v(\tau))^{p_1}\, \mathrm{d}\tau \,\mathrm{d}s +k_2 \int_{0}^{t}S_{\mathcal{G}}(t-s)(v(s))^{\tilde{p}_2} \mathrm{d}s.$$
	In fact, $u$ and $v$ are, respectively, the limit of the following Picard sequences
	\begin{equation} \label{pic1A}
		u_{1}(t) =S_\mathcal{G}(t)u_0 ,\qquad u_{k}(t) = \Phi (u_{k-1})(t),\,\,\,k\geq2 , 
	\end{equation}
	\begin{equation} \label{pic12A}
		v_{1}(t) =S_\mathcal{G}(t)u_0 ,\qquad	 v_{k}(t) = \tilde{\Phi} (v_{k-1})(t),\,\,\, k\geq2 .
	\end{equation}
	We observe that both solutions begin at the same initial data, that is, $0\leq u_1= v_1$, and $v_1(t) \leq \|u_0\|_{L^\infty}<\epsilon_0$. Since $\tilde{p}_2<p_2$, it follows that 
	$$u_1^{p_2}=v_1^{p_2}\leq  v_1^{\tilde{p}_2}.$$ 
	Hence, by Theorem \ref{comparison}, $0\leq u_2\leq v_2$. On the other hand, by the proof of {\bf Case} (i), $v$ is the global mild solution which implies, in particular, that $\{v_n\}_n\subset\Xi$, where $\Xi$ is defined in \eqref{estiF}. Recalling that $\tilde{p}_2=(p_1+1-\gamma)/(2-\gamma)$,  we estimate
	\begin{eqnarray*}
		\|v_2(t)\|_{L^{\infty}}&\leq & C t^{-\frac{N+2k}{2q_{\mathrm{sc}}}} \|u_0\|_{L^{q_{\mathrm{sc}}}} +
		C\int_{0}^{t} (t-s)^{-\frac{N+2k}{2q}p_1} \int_0^s (s-\tau)^{-\gamma} \|v_1(\tau)\|^{p_1}_{L^{q}} \, \mathrm{d}\tau \,\mathrm{d}s  \\
		&{}& +\, C \int_{0}^{t}(t-s)^{-\frac{N+2k}{2q}\tilde{p}_2} \|v_1(s)\|^{\tilde{p}_2}_{L^{q}} \mathrm{d}s \\
		&\leq& C t^{-\frac{N+2k}{2q_{\mathrm{sc}}}} \|u_0\|_{L^{q_{\mathrm{sc}}}}  + C \delta^{p_1}  t^{2-\frac{N+2k}{2q}p_1 - \gamma - \beta p_1} + C \delta^{\tilde{p}_2}  t^{1-\frac{N+2k}{2q} \tilde{p}_2-\beta \tilde{p}_2} .
	\end{eqnarray*}
	Here, $q_{sc}$ is the one defined before the statement of Theorem \ref{global1}, but associate to the exponents $p_1$ and $\tilde{p_2}$. From the choice of $q_{\mathrm{sc}}$ and $\beta$, we have
	$$
	\|v_2(t)\|_{L^{\infty}}  \leq C t^{-\frac{N+2k}{2q_{\mathrm{sc}}}} \left[\|u_0\|_{L^{q_{\mathrm{sc}}}}  +  \delta^{p_1}  + \delta^{\tilde{p}_2} \right] .
	$$
	From this, the choice of $\delta\ll 1$ and the smallness of $\|u_0\|_{L^{q_{\mathrm{sc}}}} $, one can deduce that
	\begin{equation}\label{vn1}
		\sup_{t\geq1}\|v_2(t)\|_{L^{\infty}}  \leq \epsilon_0+\delta.
	\end{equation}
	Moreover, if $t\in(0,1]$, then
	\begin{equation*}
		\|v_2(t)\|_{L^{\infty}}  \leq  \|u_0\|_{L^{\infty}}  +  \frac{\epsilon_0^{p_1}}{(1-\gamma)(2-\gamma)}  + \epsilon_0^{\tilde{p}_2}  ,
	\end{equation*}
	since $0\leq v_1(t)<\epsilon_0$. By choosing $\epsilon_0$ small enough, we get
	\begin{equation}\label{vn2}
		\sup_{0<t\leq1}\|v_2(t)\|_{L^{\infty}}  \leq 2\epsilon_0.
	\end{equation}
	Combining \eqref{vn1}-\eqref{vn2}, and choosing $\delta,\epsilon_0$ small enough, we infer that 
	$$
	\|v_{2}(t)\|_{L^{\infty}}  \leq \max\{\epsilon_0+\delta,2\epsilon_0\}<1.
	$$
	Thus, 
	$$0\leq u_2^{p_2}\leq v_2^{p_2}\leq  v_2^{\frac{p_1+1-\gamma}{2-\gamma}}.$$ 
	Hence, by Theorem \ref{comparison}, $0\leq u_3\leq v_3$. By applying the induction argument and choosing $\delta,\epsilon_0$ small enough, one can prove that $0\leq u_n\leq v_n<1$, for all $n\in\N$, whence $0\leq \|u\|_{L^{\infty}}\leq  \|v\|_{L^{\infty}}\leq 1$. Then, $\|u\|_{L^\infty((0,t)\times\mathbb{R}^{N+k})}<\infty$ for all $t>0$ and the global existence of $u$ follows from Theorem \ref{T0+}.

	\underline{The case $p_2< \tilde{p}_2$.}  Since $p_2< \tilde{p}_2$ is equivalent to $p_1 > (p_2-1)(2-\gamma)+1 = \tilde{p}_1$, then $w^{p_1}\leq w^{\tilde{p}_1}$ whenever $w<\varepsilon_0$. We would like to repeat the same argument as above by interchanging the roles of $p_1$ and $p_2$ and using $\tilde{p}_1$ instead of $\tilde{p_2}$. For this, we must guarantee that we can apply Case (i), that is, $\tilde{p_1}>p_1^*$. But this is equivalent to  $p_2>p_2^{**}$, which is our assumption. The global existence of $u$ then follows similarly.\\
	
	\noindent {\bf Case $(\mathrm{iii})$.} Suppose $k_i\leq0$. By applying the comparison principle from Theorem \ref{comparison}, we deduce that the mild solution $u$ satisfies $0\leq u(t)\leq S_{\mathcal{G}}(t)u_0$, for all $t\geq0$. As a result, $u$ is bounded from above and thus extends globally in time, yielding a global mild solution to problem \eqref{*}.\\
	
	\noindent {\bf Case $(\mathrm{iv})$.} Suppose $k_1>0$ and $k_2\leq 0$. Applying the comparison principle from Theorem \ref{comparison}, we infer that the mild solution $u$ satisfies $0\leq u\leq v$ where $v$ denotes the nonnegative mild solution of the auxiliary problem:
	\begin{equation}\label{memory}
		\left\{\begin{array}{ll}
			\,\, \displaystyle v_{t}-\Delta_{\mathcal{G}}
			v =  k_1\int_0^t (t-s)^{-\gamma}v^{p_1} \ \mathrm{d}s, & t>0,z=(x,y)\in {\mathbb{R}}^{N+k},\\
			{}\\
			\displaystyle{v(z,0)=  u_0(z),\qquad\qquad}& z=(x,y)\in {\mathbb{R}}^{N+k}.
		\end{array}
		\right.
	\end{equation}
	Using arguments analogous to those employed in Case $(\mathrm{i})$ or to the proof of Theorem $1.1$-(ii) in \cite{CDW}, it follows that $v$ is globally defined. Consequently, $u$ being bounded above by $v$, is also a global mild solution to \eqref{*}.\\
	
	\noindent {\bf Case $(\mathrm{v})$.} Now assuming $k_2>0$ and $k_1\leq0$. By again invoking the comparison principle in Theorem \ref{comparison}, we deduce $0\leq u\leq v$ where $v$ is the nonnegative mild solution for the problem 
	\begin{equation}\label{nomemory}
		\left\{\begin{array}{ll}
			\,\, \displaystyle v_{t}-\Delta_{\mathcal{G}}
			v= k_2v^{p_2}, & t>0,z=(x,y)\in {\mathbb{R}}^{N+k},\\
			{}\\
			\displaystyle{v(z,0)=  u_0(z),\qquad\qquad}& z=(x,y)\in {\mathbb{R}}^{N+k}.
		\end{array}
		\right.
	\end{equation}
	Following similar reasoning as in Case $(\mathrm{i})$ or referring to \cite[Theorem 4.2]{Oli-Vi-23}, it can be shown that $v$ exists globally in time. Hence, the original solution $u$  of \eqref{*} must also be global.\\

\end{proof}
%
%

\bibliographystyle{abbrv}
\bibliography{ref}

\end{document}